\documentclass[11pt]{article}
\usepackage[latin1]{inputenc}
\usepackage[T1]{fontenc}
\usepackage[english]{babel}
\usepackage[mathscr]{eucal}
\usepackage{amsbsy}
\usepackage{amssymb}
\usepackage{amsmath}
\usepackage{amsthm}
\usepackage{bm}
\usepackage{color}
\usepackage{graphicx}
\newcommand{\cS}{\mathcal{S}}

\newcommand{\cH}{\mathcal{H}}

\newcommand{\cU}{\mathcal{U}}
\newcommand{\eZ}{\mathbb Z}
\newcommand{\eR}{\mathbb R}
\newcommand{\eN}{\mathbb N}

\newcommand{\eB}{\mathbb B}
\newcommand{\eS}{\mathbb S}

\theoremstyle{plain}
\newtheorem{teo}{Theorem}
\newtheorem{lem}[teo]{Lemma}

\theoremstyle{definition}
\newtheorem{definition}[teo]{Definition}

\newtheorem{rem}[teo]{Remark}

\newcommand{\wc}{\xrightarrow{\mathcal{D}}}

\newcommand{\wcn}{\xrightarrow[n\to\infty]{\mathcal{D}}}

\newcommand{\pc}{\xrightarrow[n\to\infty]{\mathrm{P}}}
\newcommand{\fddc}{\xrightarrow[n\to\infty]{\mathrm{fdd}}}

\newcommand{\wt}{\widetilde}
\newcommand{\dd}{\mathrm{d}}

\newcommand{\seq}{\mathrm{seq}}

\newcommand{\ol}{\overline}
\newcommand{\var}{\mathrm{var}}

\newcommand{\flo}[1]{\lfloor #1 \rfloor}

\definecolor{darkorange}{rgb}{1.0, 0.55, 0.0}
\definecolor{royalblue4}{rgb}{0.14,0.25,0.54}

\title{Convergence of $U$-Processes in H\"older Spaces with Application to Robust Detection of a Changed Segment}

\author{Alfredas Ra\v ckauskas\footnote{The research is supported by the Research Council of Lithuania, grant No. S-MIP-17-76}\ \   and Martin Wendler}

\begin{document}

\maketitle

\begin{abstract} To detect a changed segment (so called epidemic changes) in a time series, variants of the CUSUM statistic are frequently used. However, they are sensitive to outliers in the data and do not perform well for heavy tailed data, especially when short segments get a high weight in the test statistic. We will present a robust test statistic for epidemic changes based on the Wilcoxon statistic. To study their asymptotic behavior, we prove functional limit theorems for $U$-processes in H\"older spaces. We also study the finite sample behavior via simulations and apply the statistic to a real data example.
\end{abstract}


\section{Introduction}\label{sec:intro}

In change point detection, the hypothesis is typically stationarity, but there are different types of alternatives, like the at most one change point or multiple change points. In this article, we are interested in testing stationarity with respect to the so called epidemic change or changed segment alternative: We have a random sample $X_1, X_2, \dots, X_n$ (with values in a sample space $(\eS, \cS)$ and distributions $P_{X_1}, P_{X_2}, \dots, P_{X_n}$) and we wish to test the null hypothesis
$$
H_0: P_{X_1} = P_{X_2}= \cdots = P_{X_n},
$$
versus the alternative 
\begin{align*}
H_1: \ \ &\textrm{there is a segment}\ \  I^*:=\{k^*+1, \dots, m^*\}\subset I_n:=\{1, 2, \dots, n\}\ \ \textrm{such that}\\
&P_{X_i}=\begin{cases} P \ \ &\textrm{for} \ \ i\in I_n\setminus I^*\\  Q \ \ &\textrm{for}\ \ i\in I^*,
\end{cases}\ \ \textrm{and}\ \ P\not=Q.
\end{align*}
Under $H_1$ the sample $(X_i, i\in I^*)$ constitutes a changed segment starting at $k^*$ and having the length $\ell^*=m^*-k^*$ and $Q$ is then the corresponding distribution in the changed segment. 
This type of alternative is of special relevance in epidemiology and has first been studied by Levin and Kline \cite{levin1985} in the case of a change in mean. Their test statistic is a generalization of the CUSUM (cumulated sum) statistic. Simultaneously, epidemic-type models were introduced by Commenges, Seal and Pinatel \cite{CSP} in connection with experimental neurophysiology.

If the changed segment is rather short compared to the sample size, tests that give higher weight to short segments have more power. Asymptotic critical values for such tests have been proved by Siegmund \cite{siegmund1988} in the Gaussian case (see also \cite{siegmund1995}). The logarithmic case was treated in Kabluchko and Wang \cite{kabluchko}, and the regular varying case in Mikosch and Ra\v{c}kauskas \cite{MR}. Yao \cite{yao1993} and Hu\v{s}kov\'a \cite{huskova1995} compared tests with different weightings. Ra\v ckauskas and Suquet \cite{rackauskas2004}, \cite{rackauskas2007} have suggested using a compromise weighting, that allows to express the limit distribution of the test statistic as a function of a Brownian motion. However, in order to apply the continuous mapping theorem for this statistic, it is necessary to establish the weak convergence of the partial sum process to a Brownian motion with respect to the H\"older norm.

It is well known that the CUSUM statistic is sensitive to outliers in the data, see e.g. Pr{\'a}{\v{s}}kov{\'a} and Chochola \cite{praskova2014}. The problem becomes worse if higher weights are given to shorter segments. A common strategy to obtain a robust change point test is to adapt robust two-sample tests like the Wilcoxon one. This was first used by Darkhovsky \cite{darkhovsky1976} and by Pettitt \cite{pettitt1979} in the context of detecting at most one change in a sequence of independent observations. For a comparison of different change point test see Wolfe and Schechtmann \cite{wolfe1984}. The results on the Wilcoxon type change point statistic were generalized to long range dependent time series by Dehling, Rooch, Taqqu \cite{dehling2013}. The Wilcoxon statistic can either be expressed as a rank statistic or as a (two-sample) $U$-statistic. This motivated Cs\"org\H{o} and Horv\'ath \cite{csorgo1988} to study more general $U$-statistics for change point detection, followed by Ferger \cite{ferger1994} and Gombay \cite{gombay2001}. Orasch \cite{orasch2004} and D\"oring \cite{doering2010} have studied $U$-statistics for detecting multiple change-points in a sequence of independent observations. Results for change point tests based on general two-sample $U$-statistics for short range dependent time series were given by Dehling, Fried, Garcia, Wendler \cite{dehling2015}, for long range dependent time series by Dehling, Rooch, Wendler \cite{dehling2017}. Betken \cite{betken2016} has suggested a self-normalized change-point test based on the Wilcoxon statistic. By using self-normalization, it is possible to avoid the estimation of unknown parameters in the limit distribution.

Gombay \cite{gombay1994} has suggested to use a Wilcoxon type test also for the epidemic change problem. The aim of this paper is to generalize these results in three aspects: to study more general $U$-statistics, to allow the random variable to exhibit some form of short range dependence, and to introduce weightings to the statistic. This way, we obtain a robust test which still has good power for detecting short changed segments. To obtain asymptotic critical values, we will prove a functional central limit theorem for $U$-processes in H\"older spaces.

The article is organized as follows. Section \ref{sec:test} introduces $U$-statistics type test statistics to deal with the epidemic change point problem. In Section \ref{sec:simu} some experimental results are presented and discussed whereas Section \ref{sec:data} deals with a concrete data set. Section \ref{sec:process} and Section \ref{sec:null} constitute the theoretical part of the paper where asymptotic results are established under the null hypothesis. Consistency under the alternative of a changed segment is discussed in Section \ref{sec:alt}. Finally in Section \ref{sec:crit}, we present the table with asymptotic critical values for the tests under consideration.

\section{Tests for changed segment based on $U$-statistics}\label{sec:test}

A general approach for constructing procedures to detect a changed segment is to use a measure of heterogeneity $\Delta_n(k,m)$ between two segments
$$
\{X_i, i\in I(k,m)\}\ \ \textrm{and}\ \ \{X_i, i\in I^c(k,m)\}, \ \  0\le k<m\le n,
$$ 
where $I(k,m)=\{k+1, \dots, m\}$ and $I^c(k, m)=I_n\setminus I(k, m)$.  As neither the beginning $k^*$ nor the end $m^*$ of changed segment is known, the statistics
$$
T_n:= \max_{0\le k<m\le n}\frac{1}{\rho_n(m-k)}\Delta_n(k, m)
$$
may be used to test the presence of a changed segment in the sample $(X_i)$, where $\rho_n(m-k)$ is a factor smoothing over the influence of either too short or too large data windows. In this paper we consider a class of $U$-statistic type measures of heterogeneity $\Delta_n(k, m)$ defined via a measurable function $h:\eS\times \eS\to \eR$ by
$$
\Delta_n(k, m)=\Delta_{h, n}(k, m):=\sum_{i\in I(k, m)}\sum_{j\in I_n\setminus I(k, m)}h(X_i, X_j),
$$
and the corresponding test statistics 
\begin{equation}\label{T:nh}
T_{n}(\gamma, h)=\max_{0\le k<m\le n}\frac{|\Delta_{h, n}(k, m)|}{\rho_\gamma((m-k)/n)},
\end{equation}
where $0\le \gamma<1/2$ and
$$
\rho_{\gamma}(t)=[t(1-t)]^{\gamma}, \ 0<t<1.
$$ 
Although other weighting functions are possible our choice is limited by application of a functional central limit theorem in H\"older spaces. 
 
Recall the kernel $h$ is symmetric if $h(x, y)=h(y, x)$ and antisymmetric if $h(x, y)=-h(y, x)$ for all $x, y\in \eS$. Any non symmetric kernel $h$ can be antisymmetrized by considering 
$$
\wt{h}(x, y)=h(x, y)-h(y, x), x, y\in \eS.
$$
Let's note that the kernel $h$ is antisymmetric if and only if $E[h(X,Y)]=0$ for any independent random variables with the same distribution such that the expectation exists. The if part follows by Fubini and antisymmetry. To see the only if part, first consider the one point distribution  $X=x$ and $Y=x$ almost surely to conclude that $h(x,x)=0$ for all $x$. Next, consider the two point distribution $P(X=x)=P(X=y)=1/2$ and conclude that $0=E[h(X,Y)]=(h(x,x)+h(y,y)+h(x,y)+h(y,x))/4$ and thus $h(x,y)=-h(y,x)$. So a $U$-statistic with antisymmetric kernel has expectation $0$ if the observations are independent and identically distributed and are good candidates for change point tests. We only consider antisymmetric kernels in this paper.

In the case of a real valued sample, examples of antisymmetric kernels include the CUSUM kernel $h_C(x,y)=x-y$ or the Wilcoxon kernel $h_W(x,y)=\bm{1}_{\{x<y\}}-\bm{1}_{\{y<x\}}$. The kernel $h_W$ leads to Wilcoxon type statistics
$$
T_n(\gamma, h_W):=\max_{0\le k<m\le n}\frac{1}{\rho_\gamma((m-k)/n)}\Big|\sum_{i\in I(k, m)}\sum_{j\in I_n\setminus I(k, m)}\Big[\bm{1}_{\{X_i<X_j\}}-\bm{1}_{\{X_j< X_i\}}\Big]\Big|
$$ 
whereas with the kernel $h_C$ we get CUSUM type statistics
$$
n^{-1}T_n(\gamma, h_C)=\max_{0\le k<m\le n}\frac{1}{\rho((m-k)/n)}\Big|\sum_{i=k+1}^m[X_i-\ol{X}_n]\Big|,
$$
where $\ol{X}_n:=n^{-1}\sum_{i=1}^n X_i$. As more general classes of kernels and corresponding statistics we can consider the CUSUM test of transformed data ($h(x, y):=\psi(x)-\psi(y)$) or a test based on two-sample M-estimators ($h(x,y)=\psi(x-y)$ for some monotone function, see Dehling et al. \cite{dehling2017}).

Based on invariance principles in H\"older spaces discussed in the next section, we derive the limit distribution of test statistics $T_n(\gamma, h)$. Theorems \ref{theoremA} and Theorem \ref{theoremB} provide examples of our results. Let $W=(W(t), t\ge 0)$ be a standard Wiener process and $B=(B(t), 0\le t\le 1)$ be a corresponding Brownian bridge. Define for $0\le \gamma<1/2$,
$$
T_{\gamma}:=\sup_{0\le s<t\le 1}\frac{|B(t)-B(s)|}{\rho_\gamma(t-s)}.
$$

\begin{teo}\label{theoremA} If $(X_i)_{i\in\eN}$ are independent and identically distributed random elements and $h$ is an antisymmetric kernel with $E[|h(X_1,X_2)|^{p}]<\infty$ for some $p>2$, then for any $\gamma<(p-2)/2p$, we have 
$$
\lim_{n\to\infty}P(n^{-3/2}\sigma_{h}^{-1}T_n(\gamma, h)\le x)=P(T_\gamma\le x),\ \ \textrm{for all}\ \ x\in \eR,
$$
where the variance parameter $\sigma_h$ is defined by $\sigma_h^2=\var(h_1(X_i))$ and $h_1(x)=E[h(x,X_i)]$.
\end{teo}
Note that in practice, the random variables $X_i$ might not have high moments, but if we use a bounded kernel like $h_W$, we know that the condition of the theorem holds for any $p\in(0,\infty)$, so we have the convergence for any $\gamma <1/2$. Also, in practical applications, the variance parameter has to be estimated. This can be done by
\begin{equation}\label{sigma:nh}
\hat{\sigma}^2_{n,h}:=\frac{1}{n}\sum_{i=1}^n \hat{h}_1^2(X_i)
\end{equation}
with $\hat{h}_1(x)=n^{-1}\sum_{j=1}^nh(x,X_i)$. 

For the case of a dependent sample, we consider absolutely regular sequences of random elements (also called $\beta$-mixing). Recall that the coefficients of absolute regularity $(\beta_m)_{m\in\eN}$ are defined by
\begin{equation*}
\beta_m=E\sup_{A\in\mathcal{F}_{m}^\infty}\left(P(A|\mathcal{F}_{-\infty}^{0})-P(A)\right),
\end{equation*}
where $\mathcal{F}_a^b:=\sigma(X_a,X_{a+1},\ldots,X_b)$ is the $\sigma$-field generated by $X_a,X_{a+1},\ldots,X_b$.

\begin{teo}\label{theoremB} Let $(X_i)_{i\in\eN}$ be a stationary, absolutely regular sequence and $h$ be an antisymmetric kernel, and assume that the following conditions are satisfied:
	\begin{itemize}
		\item[(i)] 
	 $\sup_{i,j\in \eN}E|h(X_i,X_j)|^{q}<\infty$ for some $q>2$; 
	 \item[(ii)] $\sum_{k=1} ^{\infty}k\beta^{1-2/q}_k<\infty$ and $\sum_k k^{r/2-1}\beta^{1-r/q}_k<\infty$ for some $2<r<q$. 
	 \end{itemize}
	 Then for any $0\le \gamma<1/2-1/r$, we have 
$$
\lim_{n\to\infty}P(n^{-3/2}\sigma_{\infty}^{-1}T_n(\gamma, h)\le x)=P(T_\gamma\le x),\ \ \textrm{for all}\ \ x\in \eR,
$$
where the long run variance parameter $\sigma_\infty$ is given by
\begin{equation*}
\sigma_\infty^2=\var\big(h_1(X_1)\big)+2\sum_{k=2}^\infty \operatorname{cov}\big(h_1(X_1),h_1(X_k)\big)
\end{equation*}
\end{teo}
For a bounded kernel $h$ the conditions (ii) on decay of the coefficients of absolute regularity reduces to 
\begin{itemize}
\item[(ii')] $\sum_k \max\{k, k^{r/2-1}\}\beta_k<\infty$ for some $r>2$.
\end{itemize}	 
Following Vogel and Wendler \cite{vogel2017}, $\sigma^2_\infty$ can be estimated using a kernel variance estimator. For this, define autocovariance estimators $\hat{\rho}(k)$ by
\begin{equation*}
\hat{\rho}(k)=\frac{1}{n}\sum_{i=1}^{n-k}\hat{h}_1(X_i)\hat{h}_1(X_{i+k})
\end{equation*}
with $\hat{h}_1(x)=\sum_{j=1}^nh(x,X_i)$. Then, for some Lipschitz continuous function $K$ with $K(0)=1$ and finite integral, we set
\begin{equation*}
\hat{\sigma}_{\infty}^2=\hat{\sigma}_h^2+2\sum_{k=1}^{n-1}K(k/b_n)\hat{\rho}(k),
\end{equation*}
where $b_n$ is a bandwidth such that $b_n\rightarrow\infty$ and $b_n/\sqrt{n}\rightarrow 0$ as $n\rightarrow \infty$.

With the help of the limit distribution and the variance estimators, we obtain critical values for our test statistic. Simulated quantiles for the limit distribution can be found in Section \ref{sec:crit}.

To discuss the behavior of the test statistics $T_{n}(\gamma, h)$ under the alternative we assume that for each $n\ge 1$ we have two probability measures $P_n$ and $Q_n$ on $(\eS, \cS)$ and a random sample $(X_{ni})_{1\le i\le n}$ such that for $k^*_n, \ell^*_n\in \{1, \dots, n\}$, 
$$
P_{X_{ni}}=\begin{cases} Q_n, \ \ &\textrm{for}\ \ i\in I^*:=\{k^*_n+1, \dots, k^*_n+\ell^*_n\}\\
P_n, \ \ &\textrm{for}\ \ i\in I_n\setminus I^*.
\end{cases}
$$
Set
$$
\delta_n=\int_{\eS}\int_{\eS} h(x, y)Q_n(dx)P_n(dy),\ \ \nu_n=\int_{\eS}\int_{\eS} (h(x, y)-\delta_n)^2Q_n(dx)P_n(dy).
$$
\begin{teo}\label{theoremC} Let $0\le \gamma<1$. Assume that for all $n\in\eN$, the random variables $X_{n1},\ldots,X_{nn}$ are independent and let $h$ be an antisymmetric kernel.
	If
\begin{equation}\label{cond:consis:1}
	\lim_{n\to\infty}\sqrt{n}\left|\delta_n\right|\Big[\frac{\ell^*_n}{n}\Big(1-\frac{\ell^*_n}{n}\Big)\Big]^{1-\gamma}=\infty
	\ \ \textrm{and}\ \ \sup_n \Big[\frac{\ell^*}{n}\Big(1-\frac{\ell^*}{n}\Big)\Big]^{1-2\gamma}\nu_n<\infty, 
	\end{equation}
then
\begin{equation}\label{consist:result}
n^{-3/2}T_n(\gamma, h)\pc \infty. 
\end{equation}
\end{teo}
For dependent random variables, we get a similar theorem:
\begin{teo}\label{theoremD} Assume that for all $n\in\eN$, the random variables $X_{n1},\ldots,X_{nn}$ are absolutely regular with mixing coefficients $(\beta_k)_{k\in\eN}$ not depending on $n$, such that $\sum_{k=1}^{\infty}k^{q/2}\beta^{1/2-1/q}_k<\infty$
for some $q>2$. Let $h$ be an antisymmetric kernel, such that there exist $C_r<\infty$ such that $E[|h(X_{in},X_{jn})|^{q}]\le C_q$ for all $n\in\eN$, $i,j\leq n$. Furthermore, let $0\le \gamma<1$ and assume that 
\begin{equation}\label{cond:consis:2}
	\lim_{n\to\infty}\sqrt{n}\left|\delta_n\right|\Big[\frac{\ell^*_n}{n}\Big(1-\frac{\ell^*_n}{n}\Big)\Big]^{1-\gamma}=\infty.
	\end{equation}
Then (\ref{consist:result}) holds. 
\end{teo}
This implies that a test based on statistic $T_n(\gamma, h)$ is consistent. More on consistency see Section \ref{sec:alt}. The proofs of Theorems \ref{theoremA} and \ref{theoremB} are given in Section \ref{sec:null}. 

\section{Simulation results}\label{sec:simu}

We compare the CUSUM type and the Wilcoxon type test statistic in a Monte Carlo simulation study. The model is an autoregressive process $(Y_n)_{n\in\eN}$ of order 1 with $Y_i=aY_{i-1}+\epsilon_i$, where $(\epsilon_i)_{i\in\eN}$ are either normal distributed, exponential distributed or $t_5$ distributed. We assume that the first $L$ observations are shifted, so that we observe
\begin{equation*}
X_i:=\begin{cases}Y_i/\sqrt{\var(Y_i)}+\delta_n\ &\text{for }i=1,\ldots,L\\Y_i/\sqrt{\var(Y_i)} \ &\text{for }i=L+1,\ldots,n\end{cases}
\end{equation*}
Under independence, the distribution of the change-point statistics does not dependent on the beginning of the changed segment, only on the length. In Table \ref{tab2}, we show some simulation results comparing the power for a changed segment in the beginning of the data and in the middle for a dependent sequence (autoregressive parameter $a=0.5$). The rejection frequencies do not differ much, so we restrict further simulations to segments of the form $I^\star=\{1,\ldots,L\}$.

\begin{table}[h]
\caption{Empirical rejection frequency under alternative for an AR(1)-process of length $N=480$ with AR-parameter $0.5$ and $t_5$ distributed-innovations, changed segment from 1 to 160 or from 161 to 320, change height $\delta_n=0.58$, level $\alpha=5\%$.}\label{tab2}
\begin{tabular}{|l||c|c|c|c|c|}
\hline 
 & \quad\quad\quad\quad\quad & \quad\quad\quad\quad\quad & \quad\quad\quad\quad\quad & \quad\quad\quad\quad\quad & \quad\quad\quad\quad\quad \\[-3mm]
 & $\gamma=0$ & $\gamma=0.1$ & $\gamma=0.2$ & $\gamma=0.3$ & $\gamma=0.4$\\[1mm]
\hline
\hline
CUSUM & & & & & \\
start=1, end=160 & 0.791 & 0.794 & 0.794 & 0.786 & 0.747\\
start=161, end=320 & 0.780 & 0,783 & 0.782 & 0.782 & 0,742\\[1mm]
\hline
Wilcoxon & & & & & \\
start=1, end=160 & 0.853 & 0.859 & 0.858 & 0.852 & 0.805\\
start=161, end=320 & 0,842 & 0,844 & 0,843 & 0,840 & 0,802\\[1mm]
\hline 
\end{tabular}
\end{table}

In Figure \ref{figpower1}, the results for $n=240$ independent observations ($a=0$) are shown. In this case, we use the known variance of our observations and do not estimate the variance. The relative rejection frequency of 3,000 simulation runs under the alternative is plotted against the relative rejection frequency under the hypothesis for theoretical significance levels of 1\%, 2.5\%, 5\% and 10\%. As expected, the CUSUM test has a better performance than the Wilcoxon test for normal distributed data. For the exponential and the $t_5$ distribution, the Wilcoxon type test has higher power. For the long changed segment ($L=80$), the weighted tests with $\gamma=0.1$ outperform the tests with $\gamma=0.3$. For the short changed segment ($L=30$), the Wilcoxon type test has more power with weight $\gamma=0.3$. The same holds for the CUSUM type test under normality. For the other two distributions however, the empirical size is also higher for $\gamma=0.3$ so that the size corrected power is not improved.

In Figure \ref{figpower2}, we show the results for $n=480$ dependent observations (AR(1) with $a=0.5$). In this case, we estimated the long run variance with a kernel estimator, using the quartic spectral kernel and the fixed bandwidth $b=4$. Both tests become too liberal now with typical rejection rates of 13\% to 15\% for a theoretical level of 10\%. For the long changed segment ($L=160$) it is better to use the weight $\gamma=0.1$, for the short segment ($L=60$) the weight $\gamma=0.3$. Under normality, the CUSUM type test has a better performance, though the difference in power is not very large. For the other two distributions, the Wilcoxon type test has a better power.

In practice, the strength of dependence is usually not known beforehand, so it would make sense to use a data-adaptive bandwidth for the variance estimation. However, the bias of the variance estimator under the alternative might get worse for data-adaptive bandwidths, and this might lead to a nonmonotonic power of change-point tests, see {e.g.} Vogelsang \cite{vogelsang1999} or Shao and Zhang \cite{shao2010}. For this reason, we propose to estimate the variance the following way: Split the data set into five shorter parts of equal length and use a variance estimator with data-adaptive bandwidth separately for each of the parts. Then take the median of the five estimators for standardizing the test statistic. The beginning and the end of the changed segment will only affect at most two of the parts, so we have at least three estimates not affected. In the simulations in Figure \ref{figpower3}, we study again an AR(1)-process and use the standard setting of the \textbf{R} function \texttt{lvar} for the data-adaptive choice of the bandwidths in the five parts. With this method, we do not observe a loss of power compared to the fixed bandwidth. Under the hypothesis, the all tests become strongly oversized. The Wilcoxon type test statistic clearly outperforms the CUSUM type statistic for nonnormal in innovations.

\begin{figure}
\includegraphics[width=0.9\textwidth]{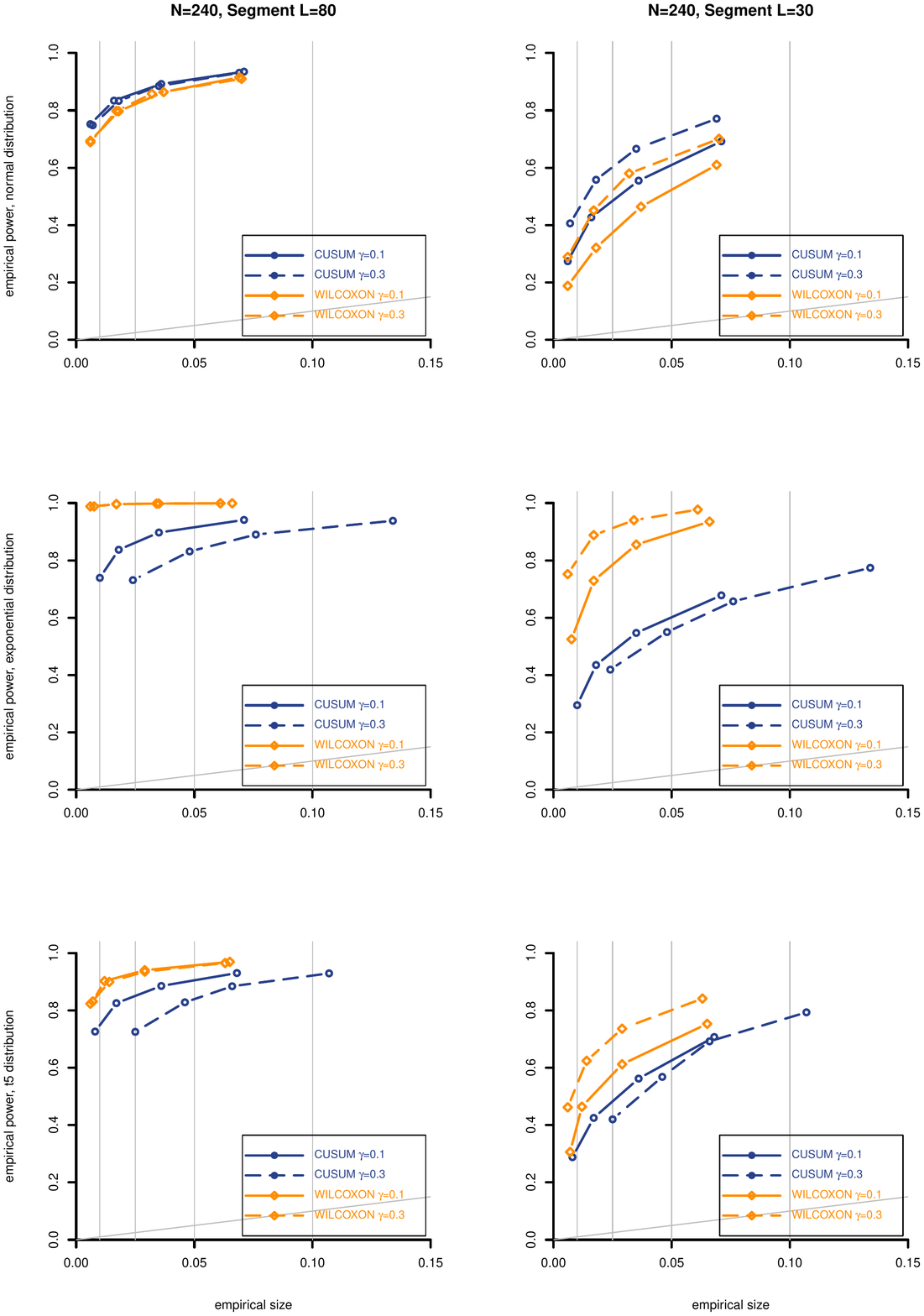} 
\caption{Rejection frequency under alternative versus rejection frequency under the hypothesis for $N=240$ independent observations using the true variance, normal (upper panels), exponential (middle panels) or $t_5$ distribution (lower panels) with change segment of length $L=80$ and height $\delta_n=0.58$ (left panels), changed segment of length $L=30$ and height $\delta_n=0.78$ (right panels), for the CUSUM type test (\textcolor{royalblue4}{$\circ$}) and for the Wilcoxon type test (\textcolor{darkorange}{$\diamond$}) with $\gamma=0.1$ (solid line) or $\gamma=0.3$ (dashed line)}
\label{figpower1}
\end{figure}

\begin{figure}
\includegraphics[width=0.9\textwidth]{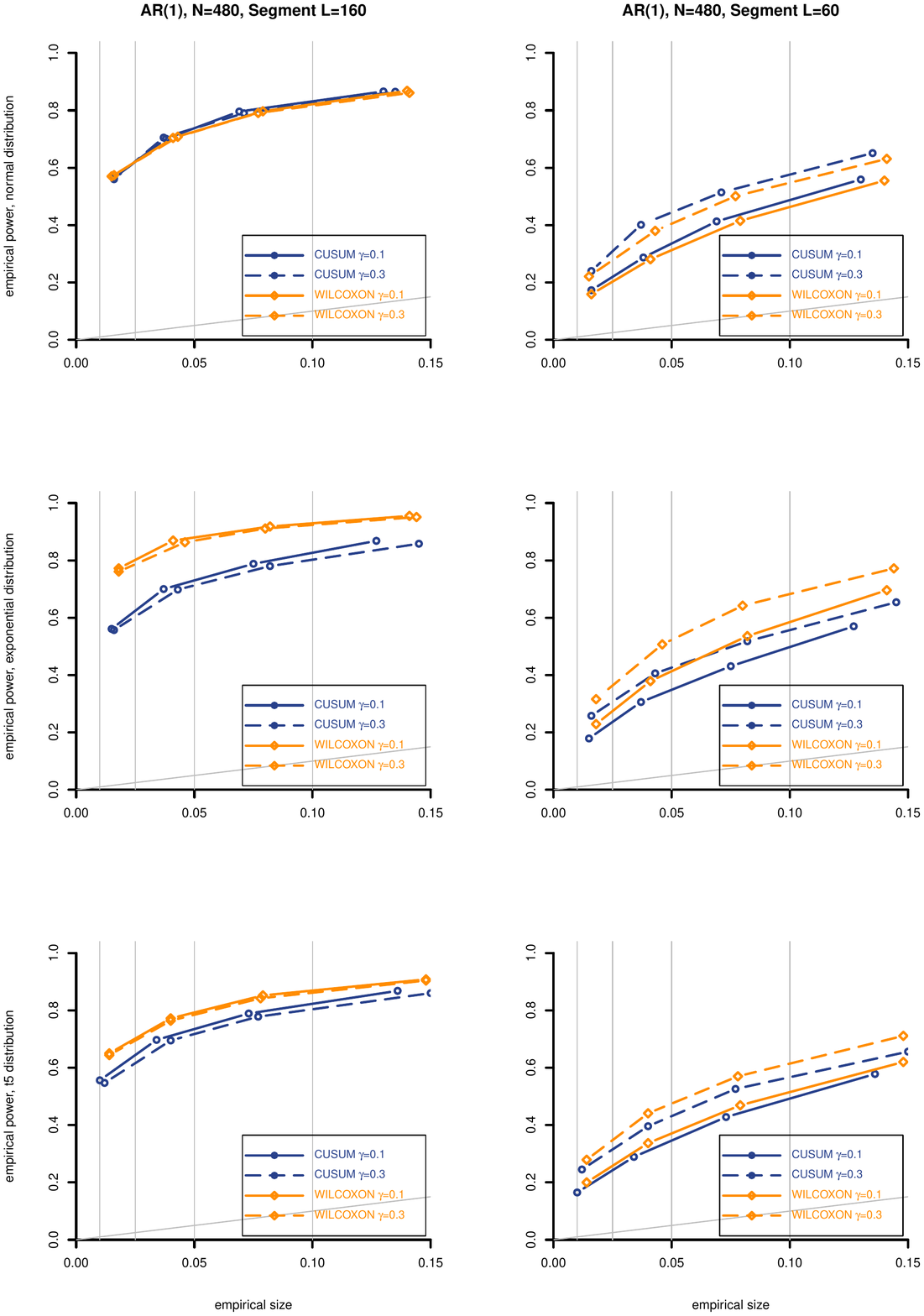} 
\caption{Rejection frequency under alternative versus rejection frequency under the hypothesis for an AR(1)-process of length $N=480$ using an estimated variance (fixed bandwidth $b_n=4$), normal (upper panels), exponential (middle panels) or $t_5$ distribution (lower panels) with changed segment of length $L=160$ and height $\delta_n=0.58$ (left panels), change segment of length $L=60$ and height $\delta_n=0.78$ (right panels), for the CUSUM type test (\textcolor{royalblue4}{$\circ$}) and for the Wilcoxon type test (\textcolor{darkorange}{$\diamond$}) with $\gamma=0.1$ (solid line) or $\gamma=0.3$ (dashed line)}
\label{figpower2}
\end{figure}

\begin{figure}
\includegraphics[width=0.9\textwidth]{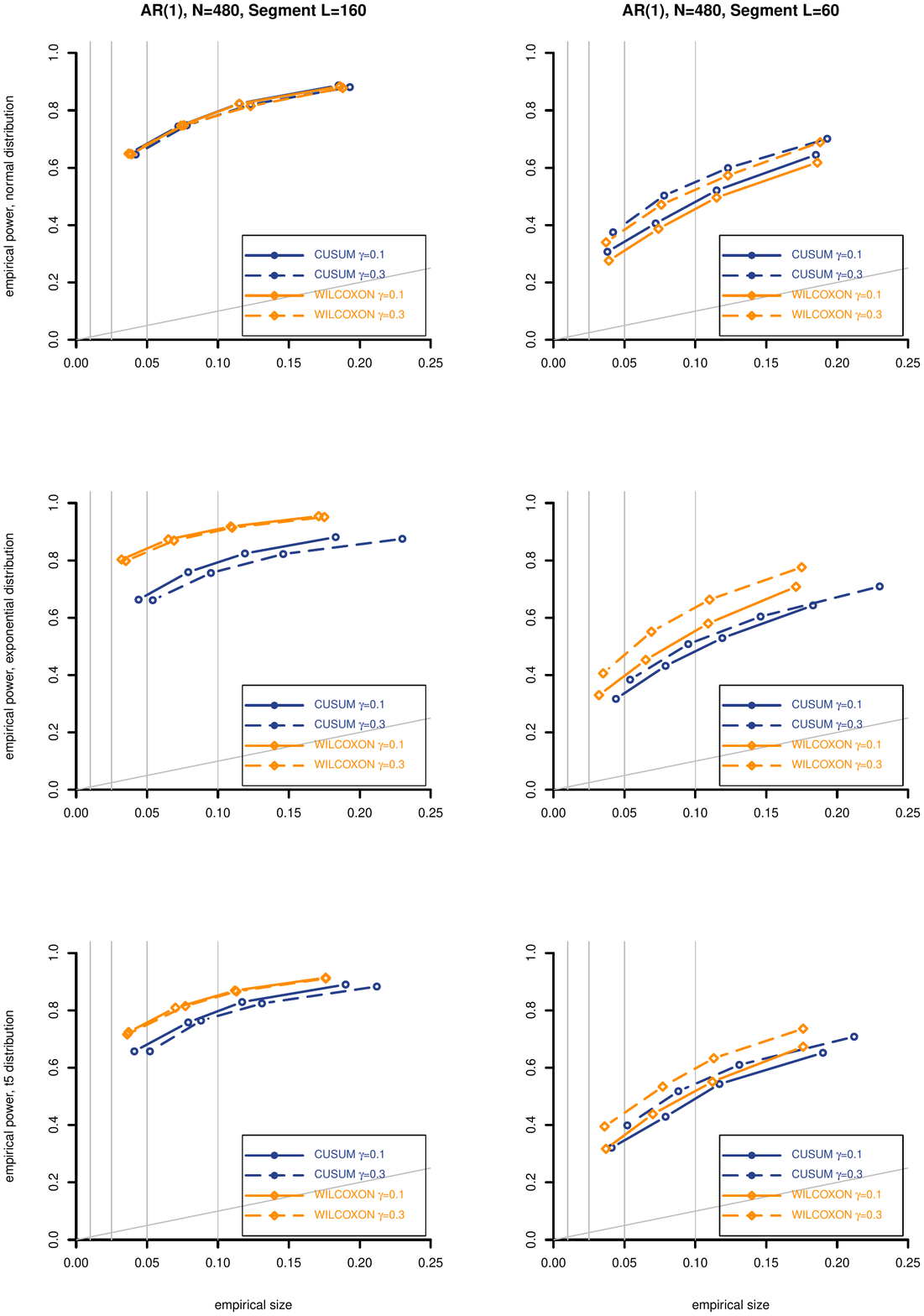} 
\caption{Rejection frequency under alternative versus rejection frequency under the hypothesis for an AR(1)-process of length $N=480$ using the median of five variance estimates with \emph{data-adaptive bandwidth}, normal (upper panels), exponential (middle panels) or $t_5$ distribution (lower panels) with changed segment of length $L=160$ and height $\delta_n=0.58$ (left panels), length $L=60$ and height $\delta_n=0.78$ (right panels), for the CUSUM type test (\textcolor{royalblue4}{$\circ$}) and for the Wilcoxon type test (\textcolor{darkorange}{$\diamond$}) with $\gamma=0.1$ (solid line) or $\gamma=0.3$ (dashed line)}
\label{figpower3}
\end{figure}

Another problem in many practical applications is the unknown length of the changed segment, so that it is difficult to choose the value $\gamma\in[0,1/2)$ to achieve the optimal power. If there is no a-priori knowledge of the typical length of an epidemic change, it would also be possible to use the maximum of (suitable standardized) test statistics for different values of $\gamma$. Another straightforward application of Theorem \ref{teo:mixing} leads to the asymptotic distribution of this combined test statistic and critical values could be obtained via simulations, but this goes beyond the scope of this paper.

\section{Data example}\label{sec:data}

We investigate the frequency of search for the term `Harry Potter' from January 2004 until February 2019 obtained from Google trends. The time series is plotted in Figure \ref{figdata}. We apply the CUSUM type and the Wilcoxon type change-point test with weight parameters $\gamma\in\{0,0.1,...,0.4\}$. The lag one autocovariance is estimated as 0.457, so that we have to allow for dependence in our testing procedure. We estimate the long run variance with a kernel estimator, using the quartic spectral kernel and the fixed bandwidth $b=4$.

The CUSUM type test does not reject the hypothesis of stationarity for a significance level of 5\%, regardless of the choice of $\gamma$. In contrast, the Wilcoxon type test detects a changed segment for any $\gamma\in\{0,0.1,...,0.4\}$, even at a significance level of 1\%. The beginning and end of the changed segment are estimated differently for different values of $\gamma$: The unweighted Wilcoxon type test with $\gamma=0$ leads to a segment from January 2008 to June 2016. For $\gamma=0.1,0.2,0.3$, we obtain January 2012 to June 2016 as an estimate. $\gamma=0.4$ leads to an estimated changed segment from January 2012 to May 2016.

By visual inspection of the time series, we come to the conclusion that the estimated changed segment for values $\gamma\geq 0.1$ fits the data better, because this segment coincides with a period with only low frequencies of search. Furthermore, the spikes of this time series can be explained by the release of movies, and the estimated changed segment is between the release of the last harry potter movie in July 2011 and the release of `Fantastic Beasts and Where to Find Them' in November 2016.

\begin{figure}
\includegraphics[width=\textwidth]{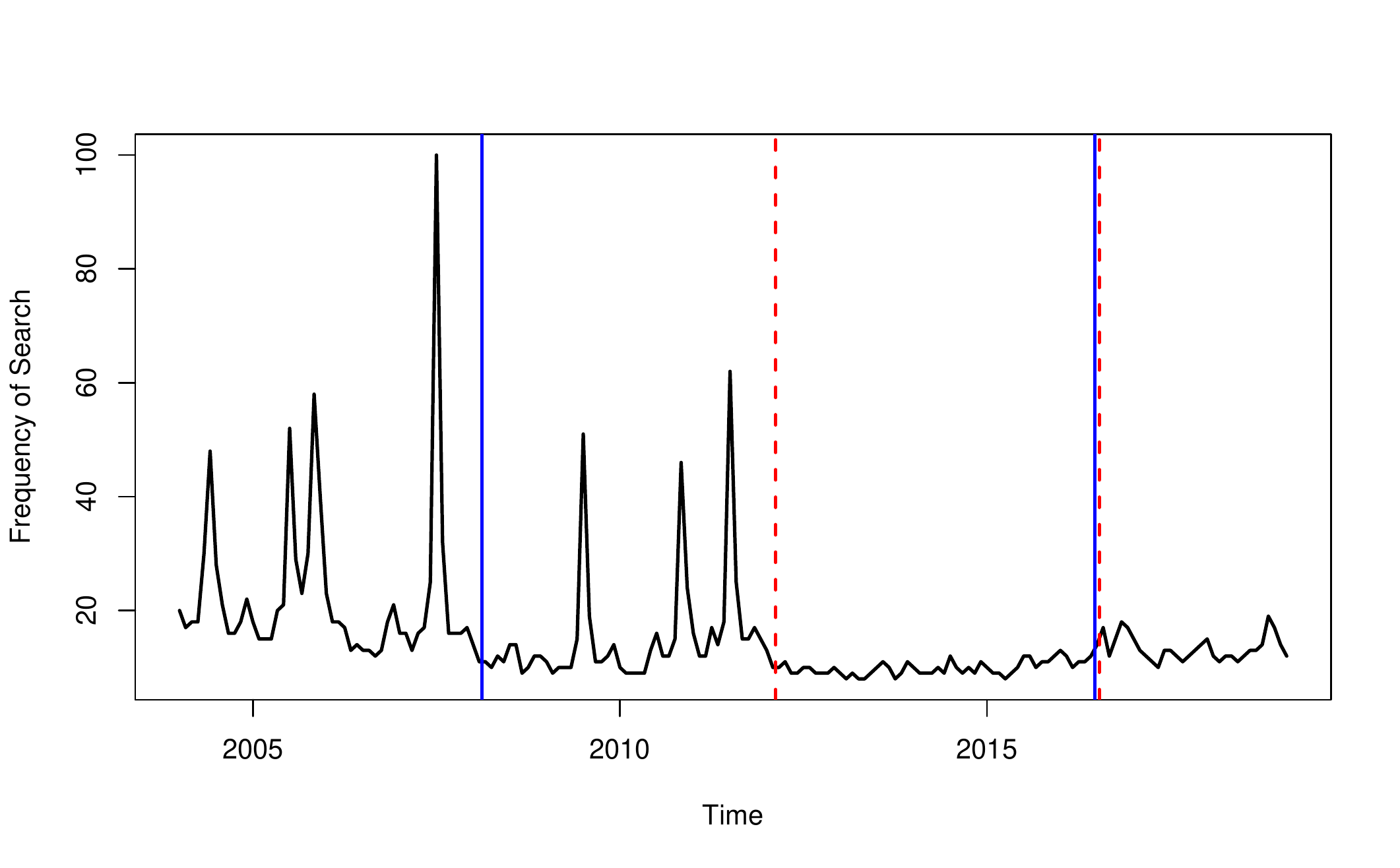} 
\caption{Frequency of search queries for `\emph{harry potter}' obtained from Google trends: CUSUM type statistic does not detect a change for any $\gamma\in\{0,0.1,0.2,0.3,0.4\}$. The Changed segment detected by the Wilcoxon type statistic with $\gamma=0$ is indicated by blue solid line, changed segment detected for $\gamma\in\{0.1,0.2,0.3\}$ by red dashed line.}
\label{figdata}
\end{figure}

\section{Double partial sum process}\label{sec:process}

Throughout this section we assume that the sequence $(X_i)$ is stationary and $P_X:=P_{X_i}$ is the distribution of each $X_i$. 
Consider for a kernel $h:\eS\times\eS\to\eR$ the double partial sums
$$
U_{h,0}=U_{h,n}=0, \ \ U_{h,k}=\sum_{i=1}^k\sum_{j=k+1}^n h(X_i, X_j),\ \  1\le k<n
$$
and the corresponding \emph{polygonal line process} $\cU_{h,n}=(\cU_{h,n}(t), t\in [0, 1])$ defined by
\begin{equation}\label{pol-line}
\cU_{h,n}(t):=U_{h,\flo{nt}}+(nt-[nt])(U_{h,\flo{nt}+1}-U_{h,\flo{nt}}),\ \ t\in [0, 1],
\end{equation}
where for a real number $a\geq 0$, $\lfloor a\rfloor:=\max\{k\colon\, k\in\eN,\,k\leq x\}$, $\eN=\{0,1,\dots\}$, is a value of the floor function.
So $\cU_{h,n}=(\cU_{h,n}(t), t\in [0, 1])$, is a random polygonal line with vertexes $(U_{h, k}, k/n)$, $k=0, 1, \dots, n$. As a functional framework for the process $\cU_{h,n}$ we consider Banach spaces of H\"{o}lder functions. Recall the space $C[0, 1]$ of continuous functions on $[0, 1]$ is endowed with the norm
$$
||x||=\max_{0\le t\le 1}|x(t)|.
$$
The H\"{o}lder space $\cH_\gamma^{o}[0, 1], 0\le\gamma<1$, of functions $x\in C[0, 1]$ such that
$$
\omega_\gamma(x, \delta):=\sup_{0<|s-t|\le \delta}\frac{|x(t)-x(s)|}{|t-s|^\gamma}\to 0\ \ \textrm{as}\ \ \delta\to 0,
$$ 
is endowed with the norm
$$
||x||_\gamma:=|x(0)|+\omega_\gamma(x, 1).
$$
Both $C[0, 1]$ and $\cH^o_\gamma[0, 1]$ are separable Banach spaces. The space $\cH^o_0[0, 1]$ is isomorphic to $C[0, 1]$.

\begin{definition} For a kernel $h$ and a number $0\le \gamma<1$ we say that $(X_i)$ satisfies $(h, \gamma)$-FCLT if there is a Gaussian process $\cU_h=(\cU_h(t), t\in [0, 1])\in \cH^o_\gamma[0, 1]$, such that 
	$$
	n^{-3/2}\cU_{h, n}\wcn \cU_h
	\ \ \textrm{in the space}\ \  \cH_\gamma^o[0, 1].
	$$ 	 
\end{definition}
In order to make use of results for partial sum processes, we decompose the $U$-statistics into a linear part and a so-called degenerate part. Hoeffding's decomposition of the kernel $h$  reads 
$$
h(x, y)=h_1(x) - h_1(y) + g(x, y),\ \ x, y\in \eS,
$$
where
$$
h_1(x)=\int_{\eS} h(x, y)P_X(\dd y),\ \ \textrm{and}\ \ g(x, y)=h(x, y)-h_1(x)+h_1(y), \ \  x, y\in \eS, 
$$
and leads to the splitting
\begin{equation}\label{H-eq:1}
\cU_{h, n}(t)=n[W_{h_1,n}(t)-tW_{h_1, n}(1)]+\cU_{g, n}(t),\ \ t\in [0, 1], 
\end{equation}
where  
$$
W_{h_1,n}(t)=\sum_{i=1}^{\flo{nt}}h_1(X_i)+
(nt-\flo{nt})h_1(X_{\flo{nt}+1}),\ \ t\in [0, 1],
$$
is the polygonal line process defined by partial sums of random variables $(h_1(X_i))$. 
Decomposition (\ref{H-eq:1}) reduces $(h, \gamma)$-FCLT to H\"olderian invariance principle for random variables $(h_1(X_i))$ via the following lemma.

\begin{lem}\label{H-lem:3} If there exists a constant $C>0$ 
such that for any integers $0\le k<m\le n$
\begin{equation}\label{cond:2}
E(U_{g,m}-U_{g,k})^2\le C(m-k)(n-(m-k))
\end{equation}
then 
$$
||n^{-3/2}\cU_{g, n}||_{\gamma}=o_P(1)
$$
for any $0\le \gamma<1/2$.
\end{lem}

\begin{rem} For an antisymmetric kernel $h$ the condition (\ref{cond:2}) follows from the following one: there exists a constant $C>0$ such that for any $0 \le m_1 <  n_1 \le m_2 < n_2$,
	\begin{equation}\label{cond:2a}
	E\Big(\sum^{n_1}_{i=m_1+1}\sum^{n_2}_{j=m_2+1}g(X_i, X_j)\Big)^2\le 
	C(n_1 - m_1)(n_2 - m_2).
	\end{equation}
Indeed, by antisymmetry 
$$
U_{g,m}-U_{g,k}=\sum_{i=k+1}^m\sum_{j=m+1}^n h(X_i, X_j)+\sum_{i=k+1}^m\sum_{j=1}^k h(X_i, X_j), 
$$
so that (\ref{cond:2a}) yields
$$
E(U_{g,m}-U_{g,k})^2\le 2C[(m-k)(n-m)+(m-k)(k-1)]\le 2C(m-k)(n-(m-k)).
$$
\end{rem}

Before we proceed with the proof of Lemma \ref{H-lem:3} we need some preparation.	
Let $D_j$ be the set of dyadic numbers of level $j$ in $[0,1]$,
that is $D_0 :=\{0,1\}$ and for $j\geq 1$, $D_j:=
\bigl\{(2l-1)2^{-j};\;1\leq l \leq 2^{j-1}\bigr\}$. For $r\in D_j$
set $r^-:=r-2^{-j}$, $r^+:=r+2^{-j}$, $j\geq 0$.
For $f:[0,1]\to \eR$ and $r\in D_j$ define
\[
\lambda_r(f):= 
\begin{cases}
f(r^+)+f(r^-)-2f(r) &\text{if $j\ge 1$,}\\
f(r)                &\text{if $j=0$.}
\end{cases}
\]
The following sequential norm on $\cH^o_\gamma[0, 1]$ defined by
$$
2^{-1}||f||^{\seq}_\gamma:=\sup_{j\ge 0}2^{\gamma j}\max_{r\in D_j}|\lambda_r(f)|,
$$
is equivalent to the norm $||f||_\gamma$, 
see~\cite{Ciesielski:60}: there is a positive constant $c_\gamma$ such that
\begin{equation}\label{equiv:1}
||f||^{\seq}_\gamma\le ||f||_\gamma\le c_\gamma||f||^{\seq}_\gamma,\ \ \ f\in \cH^o_\gamma[0, 1].
\end{equation}  
Set $\mathcal{D}_j:=\{k2^{-j},$ $0\le k\le 2^j\}$.
In what follows, we denote by $\log$ the logarithm with basis $2$
($\log 2 = 1$).

\begin{lem}\label{lemma1} For any $0\le \gamma\le 1$ there is a constant $c_\gamma>0$ such that, if $V_n$ is a polygonal line function with vertexes $(0, 0), (k/n, V_n(k/n)), k=1, \dots, n$, then 
	$$
	||V_{n}||_\gamma\le  c_\gamma \max_{0\le j\le \log n}2^{\gamma j}\max_{r\in \mathcal{D}_j}\Big|V_{n}(\flo{nr+n2^{-j}}/n)-V_{n}(\flo{nr}/n)\Big|.
	$$
\end{lem}
\begin{proof}
	First we remark that for any $j\ge 1$,
	$$
	\max_{r\in D_j}|\lambda_r(V_n)| \le
	\max_{r\in D_j}|V_n(r^+)-V_n(r)| + \max_{r\in D_j}|V_n(r)-V_n(r^-)|.
	$$
	As $r^+$ and $r^-$ belong to $\mathcal{D}_j$, this gives, 
	$$
	\sup_{j\ge 1}2^{\gamma j}\max_{r\in D_j}|\lambda_r(V_n)|
	\le
	2\sup_{j\ge 1}2^{\gamma j}\max_{r\in \mathcal{D}_j}|V_n(r+2^{-j})-V_n(r)|
	$$
and it follows by (\ref{equiv:1}),  
	$$
	||V_n||_\gamma \le 2 c_\gamma\sup_{j\ge 0}2^{\gamma j}\max_{r\in \mathcal{D}_j}|V_n(r+2^{-j})-V_n(r)|.
	$$
If $s$ and $t>s$ belong to the same interval, say, $[(k-1)/n, k/n]$, then, observing that the slope of $V_n$ in this interval is precisely $n[V_n(k/n)-V_n((k-1)/n)]$, we have
\begin{align*}
|V_n(t)-V_n(s)|&=
n(t-s)|V_n(k/n)-V_n((k-1)/n)|\le n(t-s)\Delta_n, 
\end{align*}
where $\Delta_n=\max_{1\le k\le n}|V_n(k/n)-V_n((k-1)/n)|$.
If $s\in [(k-1)/n, k/n), t\in [k/n, (k+1)/n)$ then 
\begin{align*}
|V_n(t)-V_n(s)|&\le |V_n(t)-V_n(k/n)|+|V_n(k/n)-V_n(s)|
\le n(t-s)\Delta_n. 
\end{align*}	
If $s\in [(k-1)/n, k/n)$, $t\in [(j-1)/n, j/n)$ and $j>k+1$, then
\begin{align*}
|V_n(t)-V_n(s)|&\le |V_n(t)-V_n((j-1)/n)|+|V_n(k/n)-V_n((j-1)/n)|+|V_n(k/n)-V_n(s)|\\
&\le |V_n(k/n)-V_n((j-1)/n)|+n[(k/n-s)+(t-(j-1)/n)]\Delta_n.
\end{align*}	
We apply these three configurations to $s=r$ and $t=r+2^{-j}$. 
If $j\ge \log n$ then only the first two configurations are possible and we deduce
\begin{align*}
\max_{j\ge \log n}2^{\gamma j}\max_{r\in \mathcal{D}_j}|V_n(r+2^{-j})-V_n(r)|
\le \max_{j\ge \log n}2^{\gamma j}n2^{-j}\Delta_n= 
2n^\gamma \Delta_n. 
\end{align*}
If $j<\log n$ then we apply the third configuration to obtain
\begin{align*}
\max_{j<\log n}2^{\gamma j}&\max_{r\in \mathcal{D}_j}|V_n(r+2^{-j})-V_n(r)|
\le \max_{j< \log n}2^{\gamma j}\max_{r\in \mathcal{D}_j}|V_n(\flo{nr+n2^{-j}}/n)-V_n(\flo{nr}/n)|\\
&+
2\max_{j<\log n}2^{\gamma j}n2^{-j}\max_{1\le k\le n}|V_n(k/n)-V_n((k-1)/n)|\\
&\le \max_{j< \log n}2^{\gamma j}\max_{r\in \mathcal{D}_j}|V_n(\flo{nr+n2^{-j}}/n)-V_n(\flo{nr}/n)| +
2n^\gamma\Delta_n.
\end{align*}
To complete the proof just observe that $\flo{nr+2^{-j}}=\flo{nr}+1$ if $j=\log n$ and so $\Delta_n\le \max_{j\le \log n}2^{\gamma j}\max_{r\in \mathcal{D}_j}|V_n(\flo{nr+n2^{-j}}/n)-V_n(\flo{nr}/n)|$.  
\end{proof}

\begin{proof}[Proof of Lemma \ref{H-lem:3}] By Lemma \ref{lemma1} we have with some constant $C>0$, 
$$
E||\cU_{g,n}||^2_\gamma\le C\sum_{j=0}^{\log n}2^{2\gamma j}2^j\max_{r\in \mathcal{D}_j}E\Big(\cU_{g,n}(\flo{nr+n2^{-j}}/n)-
\cU_{g,n}(\flo{nr}/n)\Big)^2.
$$
Condition (\ref{cond:2}) gives
$$
E\Big(\cU_{g, n}(m/n)-\cU_{g, n}(k/n)\Big)^2\le C(m-k)(n-(m-k)).
$$
This yields, taking into account that $\flo{nr+n2^{-j}}-\flo{nr}\le n2^{-j}$ for $r\in \mathcal{D}_j$, 
$$
E||n^{-3/2}\cU_{g,n}||^{2}_\gamma\le C_\gamma n^{-3}\sum_{j=1}^{\log n}2^{2\gamma j} 2^j[n2^{-j}(n-n2^{-j})]\le C_\gamma n^{-1+2\gamma}.
$$
This completes the proof due to the restriction $0\le \gamma<1/2$.  
\end{proof}

The next lemma gives a general conditions for the tightness of the sequence $(n^{-1/2}W_{h_1, n})$ in H\"{o}lder spaces.

\begin{lem}\label{H-lem:4} Assume that the sequence $(X_i)_{i\in\eN}$ is a stationary and for a $q> 2$, there is a constant $c_q>0$ such that for any $0\le k<m\le n$
	\begin{equation}\label{Ros:1}
	E\Big|\sum_{i=k+1}^m h_1(X_i)\Big|^q\le c_q(m-k)^{q/2}.
\end{equation} 
Then for any $0\le \gamma<1/2-1/q$ the sequence $(n^{-1/2}W_{h_1,n})$ is tight in the space $\cH^o_{\gamma}[0, 1]$.
\end{lem}

\begin{proof} Fix $\beta>0$ such that $0\le \gamma<\beta<1/2-1/q$.
By Arcela-Ascoli the embedding $\cH^o_\beta[0, 1]\to \cH^o_\gamma[0, 1]$ is compact, hence, it is enough to prove 
	\begin{equation}\label{tight_eq:1}
	\lim_{a\to\infty}\sup_{n\ge 1}P(||n^{-1/2}W_{h_1,n}||_\beta>a)=0.
	\end{equation}
By Lemma \ref{lemma1},
$$
P(||n^{-1/2}W_{h_1,n}||_\beta>a)\le I_n(a),
$$
where
$$
I_{n}(a)=P\Big(\max_{0\le j\le \log n}2^{\beta j}\max_{r\in \mathcal{D}_j}\Big|W_{h_1,n}(\flo{nr+n2^{-j}}/n)-W_{h_1,n}(\flo{nr}/n)\Big|\ge c_\beta n^{1/2}a\Big).
$$
with some constant $c_\beta>0$. 
Since $\flo{nr+n2^{-j}}-\flo{nr}\le n2^{-j}$ we have by condition (\ref{Ros:1}),
\begin{align*}
I_n(a)&\le cn^{-q/2}a^{-q}\sum_{j=1}^{\log n}2^{q\beta j}2^j\max_{r\in D_j}E\Big|W_{h_1, n}(\flo{nr+n2^{-j}}/n)-W_{h_1, n}(\flo{nr}/n)\Big|^q\\
&=cn^{-q/2}a^{-q}\sum_{j=1}^{\log n}2^{q\beta j}2^j\max_{r\in D_j}E\Big|\sum^{\flo{nr+n2^{-j}}}_{i=\flo{nr}+1}h_1(X_i)\Big|^q\\
&\le cn^{-q/2}a^{-q}\sum_{j=1}^{\log n}2^{q\beta j}2^j (n2^{-j})^{q/2}\\
&\le ca^{-q}\sum_{j=1}^{\log n} 2^{-j(q/2-q\beta-1)},
\end{align*}
with some constant $c>0$. Since $q/2-q\beta-1>0$, we obtain $I_n(a)\le ca^{-q}$ and complete the proof of (\ref{tight_eq:1}) and that of the lemma.   
\end{proof}

Summing up we have the following functional limit theorem for the process $\cU_{h,n}$.

\begin{teo}\label{general:FCLT}
Assume that the sequence $(X_i)$ is stationary sequence of $\eS$-valued random elements. Let $h$ be an antisymmetric kernel end $E|h(X_1, X_2)|^p<\infty$ for some $p>2$. If   
	\begin{itemize}
		\item[(i)] there is a constant $C>0$ such that for any $0\le m_1<n_1\le m_2<n_2$ the inequality (\ref{cond:2a}) is satisfied;
		\item[(ii)] for some $2<q\le p$ the inequality (\ref{Ros:1}) is satisfied;
		\item[(iii)] there is a Gaussian process $\cU_h$ such that
		$$
		n^{-1/2}W_{h_1,n}\fddc \cU_h,
		$$
	\end{itemize}
then 
$$
n^{-3/2}\cU_{h,n}\wc \cU^o_h\ \ \textrm{in the space}\ \ \cH^o_\gamma[0, 1]
$$
for any $0\le \gamma<1/q,$ where $\cU_h^o=(\cU_h(t)-t\cU_h(1), t\in [0, 1])$.	
\end{teo}

\subsection{iid sample}
In this subsection we establish the $(h, \gamma)-FCLT$ for independent identically distributed sequences $(X_i)_{i\in\eN}$.

\begin{teo}\label{iidsample} Assume that $(X_i)$ are independent and identically distributed random elements in $\eS$ and the measurable function $h:\eS\times\eS\to\eR$ is antisymmetric. If $E|h(X_1, X_2)|^q<\infty$ for some $q>2$, then $(X_i)$ satisfies $(h, \gamma)-FCLT$ for any $0\le \gamma <1/2-1/q$ with the limit process $\cU_{h_1}=\sigma_{h} B$, where $B=(B(t), t\in [0, 1])$ is a standard Brownian bridge.

Particularly, if the kernel $h$ is antisymmetric and bounded, then
$(X_i)$ satisfies $(h, \gamma)$-FCLT for any $0\le \gamma<1/2$.
\end{teo}

\begin{proof} We need to check conditions (i)-(iii) of Theorem \ref{general:FCLT}. Starting with (i) we have

$$
E\Big(\sum^{n_1}_{i=m_1+1}\sum^{n_2}_{j=m_2+1}g(X_i, X_j)\Big)^2
=\sum_{i, i'=m_1+1}^{n_1}\sum_{j=m_2+1}^{n_2}Eg(X_i, X_j)g(X_{i'}, X_{j'})
$$
and observe that $Eg(X_i, X_j)g(X_{i'}, X_{j'})=0$ if either $i\not=i'$ or $j\not=j'$. Indeed, it is enough to observe that $Eg(X_1, x)=0$ for each $x$:
\begin{align*}
Eg(X_1, x)&=E[h(X_1, x)-h_1(X_1)+h_1(x)]\\
&=E[-h(x, X_1)-h_1(X_1)+h_1(x)]\\
&=Eh_1(X_1)=0.
\end{align*}
Now, if $i\not=i'$, $j=j'$ then we have
$$
Eg(X_i, X_j)g(X_{i'}, X_{j'})=\int_{\eS}Eg(X_i, x)Eg(X_{i'}, x)P_X(dx)=0.
$$
Hence,
\begin{align*}
E\Big(\sum^{n_1}_{i=m_1+1}\sum^{n_2}_{j=m_2+1}g(X_i, X_j)\Big)^2&=
\sum^{n_1}_{i=m_1+1}\sum^{n_2}_{j=m_2+1}Eg^2(X_i, X_j)\\
&=(n_1-m_1)(n_2-m_2)Eg^2(X_1, X_2)\\
&\le 4(n_1-m_1)(n_2-m_2)Eh^2(X_1, X_2).
\end{align*}	
Condition (ii) is obtained via 
Rosenthal's inequality. Since the moment assumption gives
	$E|h_1(X_1)|^q=E[|E[h(X_1, X_2)|X_2]|^q]\le E|h(X_1, X_2)|^q<\infty$ we have 
	\begin{align*}
E\Big|\sum_{i=k+1}^m h_1(X_i)\Big|^q&\le c_q\Big[\Big(\sum_{i=k+1}^m Eh_1^2(X_i)\Big)^{q/2}+\sum_{i=k+1}^m E|h_1(X_i)|^q\Big]\\
&\le 2c_q(m-k)^{q/2}E|h_1(X_1)|^q.
\end{align*}
As the convergence $n^{-1/2}W_{h_1, n}\fddc \sigma_{h_1}W $ is well known, the proof is completed. 
\end{proof}

\subsection{Mixing sample}

In this subsection we establish the $(h, \gamma)-FCLT$ for $\beta$-mixing sequences $(X_i)_{i\in\eN}$. For $A\subset \eZ$ we will denote by $P_A$ the joint distribution of $\{X_i, i\in A\}$. We write $P_X$ for the distribution of $X_i$. We need some auxiliary lemmas:

\begin{lem}\label{g_2} Let $i_1<i_2<\cdots<i_k$ be arbitrary integers. Let $f:\eS^k\to \eR$ be a measurable function such that for any $j$, $1\le j\le k-1$,
	$$
	\int_{\eS^k} |f|^{1+\delta} d\Big[P_{{i_1},\dots, {i_k}}+P_{{i_1}, \dots, {i_j}}\otimes P_{{i_{j+1}}, \dots, {i_k}}\Big]<M,
	$$
	for some $\delta>0$. Then 
	$$
	\Big|\int_{\eS^k}f \ d\Big(P_{X_{i_1},\dots, X_{i_k}}-P_{X_{i_1}, \dots, X_{i_j}}\otimes P_{X_{i_{j+1}}, \dots, X_{i_k}})\Big|\le 4M^{1/(1+\delta)}\beta^{\delta/(1+\delta)}_{i_{j+1}-i_j}.
	$$
\end{lem} 	
\begin{proof} The proof goes along the lines of the proof of Lemma 1 in Ken-ichi Yoshihara \cite{yoshihara1976}.
\end{proof}	

\begin{lem}\label{g2b} Assume that for a $\delta>0$ there is a constant $M$ such that 
	$$
	E|h(X_i, X_j)|^{2(1+\delta)}\le M
	$$
	for any $1\le i, j\le n$ and 
	$$
	\sum_{k=0}^\infty k\beta^{\delta/(1+\delta)}_k<\infty.
	$$
Then for any $0\le m_1<n_1\le m_2<n_2$,
	$$
	I(m_1, n_1, m_2,n_2):=E\Big(\sum_{i=m_1+1}^{n_1}\sum_{j=m_2+1}^{n_2}g(X_i, X_j)\Big)^2\le C(n_1-m_1)(n_2-m_2)
	$$
\end{lem}
\begin{proof}
We have
$$
I(m_1, n_1, m_2,n_2)=\sum_{i_1, i_2=m_1+1}^{n_1}\sum_{j_1, j_2=m_2+1}^{n_2}J(i_1, i_2, j_1, j_2),
$$
where
$$
J(i_1, i_2, j_1, j_2)=Eg(X_{i_1}, X_{j_1})g(X_{i_2}, X_{j_2}).
$$
First consider the case where $i_1<i_2$ and $j_1<j_2$. If $j_2-j_1>i_2-i_1$ then by Lemma \ref{g_2} we have
$$
\Big|J(i_1, i_2, j_1, j_2)-\int_{\eS^4}g(x_1, x_2)g(x_3, x_4)dP_{X_{i_1}, X_{i_2}, X_{j_1}}\otimes P_{X_{j_2}}\Big|\le 4M^{1/(1+\delta)}\beta^{\delta/(1+\delta)}_{j_1-j_2}.
$$
If $i_2-i_1>j_2-j_1$ then
$$
\Big|J(i_1, i_2, j_1, j_2)-\int_{\eS^4}g(x_1, x_2)g(x_3, x_4)dP_{X_{i_1}}\otimes P_{X_{i_2}, X_{j_1}, X_{j_2}}\Big|\le 4M^{1/(1+\delta)}\beta^{\delta/(1+\delta)}_{i_1-i_2}.
	$$
Note that for any $y\in \eS$,
$$
\int_{\eS} g(y, x)P_{X_{j_2}}(dx)=\int_{\eS} [h(y,x)-(h_1(y)-h_1(x))]P_{X_{j_2}}(dx)=0
$$
and
$$
\int_{\eS} g(x, y)P_{X_{i_1}}(dx)=\int_{\eS} [h(x,y)-(h_1(x)-h_1(y))]P_{X_{i_1}}(dx)=0.
$$
Treating the other cases in the same way, we deduce that for any $m_1<i_1,i_2\le n_2\le m_2<j_1,j_2\le n_2$, 
$$
|J(i_1, i_2, j_1, j_2)|\le 4M^{1/(1+\delta)}\beta^{\delta/(1+\delta)}_{\min\{|i_2-i_1|,|j_2-j_1|\}}.
$$
This yields
\begin{align*}
\left|I(m_1, n_1, m_2,n_2)\right|\leq C\sum_{i_1, i_2=m_1+1}^{n_1}\sum_{j_1, j_2=m_2+1}^{n_2}\beta^{\delta/(1+\delta)}_{\min\{|i_2-i_1|,|j_2-j_1|\}}.
\end{align*}
If $k:=\min\{|i_2-i_1|,|j_2-j_1|\}=|i_2-i_1|$, then there are less than $n_1-m_1$ choices for $i_1$, at most 2 choices for $i_2$, as $i_2\in\{i_1-k,i_1+k\}$. Furthermore, there are less than $n_2-m_2$ choices for $j_1$, and, because $|j_2-j_1|\leq k$, at most $2k+1$ choices for $j_2$. In the case $k:=\min\{|i_2-i_1|,|j_2-j_1|\}=|j_2-j_1|$, we can use a similar reasoning. In total, there are less than $12(n_1-m_1)(n_2-m_2)k$ ways to chose the indices for given $k$. We arrive at
\begin{equation*}
\left|I(m_1, n_1, m_2,n_2)\right|\leq C(n_1-m_1)(n_2-m_2)\sum_{k=0}^\infty k\beta^{\delta/(1+\delta)}_k=C(n_1-m_1)(n_2-m_2)
\end{equation*}

provided that $\sum_k k\beta^{\delta/(1+\delta)}_k<\infty$.

\end{proof}

\begin{lem}\label{beta:Ros} Assume that
	$$
	\int_{\eS}\Big(\int_{\eS}h(x, y)P_X(dy)\Big)^{r+\delta}P_X(dx)<\infty
	$$ 
	for some $r>2$ and $\delta>0$. If
	$$
	\sum_k k^{r/2-1}\beta^{\delta/(r+\delta)}_k<\infty
	$$
	then there is a constant $c_{r, \delta}>0$ such that for any $0\le k<m\le n$,
	$$
	E\Big|\sum_{i=k+1}^m h_1(X_i)\Big|^r\le c_{r, \delta} (m-k)^{r/2}.
	$$
\end{lem}

\begin{proof}
This lemma is proved in Yokoyama \cite{Yokoyama} for real valued strongly mixing random variables. We need to note that if $(X_i)$ is $\beta$-mixing then $(h_1(X_i))$ is $\beta$-mixing as well for any measurable $h_1:\eS\to \eR$. Being such this sequence is also strongly mixing.
\end{proof}

\begin{teo}\label{teo:mixing}
Assume that $(X_i)$ is a strictly stationary $\beta$-mixing sequence of random elements in $\eS$ and the measurable function $h:\eS\times\eS\to\eR$ is antisymmetric. If $E|h(X_1, X_2)|^{q}<\infty$ and 
	\begin{equation}\label{mixing:cond1}
	\sum_k k\beta^{1-2/q}_k<\infty,\ \ \sum_k k^{r/2-1}\beta^{1-r/q}_k<\infty,
	\end{equation}
	for some $q>2$ and $2<r<q$, then 
	$(X_i)$ satisfies $(h, \gamma)-FCLT$ for any $0\le \gamma <1/2-1/r$ with the limit process $\cU_h=\sigma_\infty B$, where $B=(B(t), t\in [0, 1])$ is a standard Brownian bridge and
	$$
	\sigma^2_\infty=\var\big(h_1(X_1)\big)+2\sum_{k=2}^\infty \operatorname{cov}\big(h_1(X_1),h_1(X_k)\big).
	$$
	Particularly, if the kernel $h$ is antisymmetric and bounded then condition (\ref{mixing:cond1}) becomes $\sum_k k^{r/2-1}\beta_k<\infty$,
and in this case $(X_i)$ satisfies $(h, \gamma)$-FCLT for any $0\le \gamma<1/2-1/r$.
\end{teo}	 
\begin{proof} We need to check conditions (i)-(iii) of Theorem \ref{general:FCLT}. First we check (i) using Lemma \ref{g2b} with $\delta=(q-2)/2$. Condition (ii) follows imediately from Lemma \ref{beta:Ros}. Finally, convergence of finite dimensional distributions can be deduced from invariance principles for $\alpha$-mixing sequences proved by a number of authors (see, e.g., \cite{herndorf} and references therein). 	
\end{proof}

 \section{Asymptotic distribution under null}\label{sec:null}

In the following, we show how the asymptotic behaviour of the statistic $T_{n}(\gamma, h)$ follows from the functional limit results for $U$-processes:

	 \begin{teo}\label{teo1} Let $0\le \gamma<1/2$ and let the kernel $h:\eS\times\eS\to \eR$ be antisymmetric. Assume that $(X_i)$ is a stationary sequence and satisfies $(h, \gamma)$-FCLT with the limit process $\cU_h$. Then
	 	$$ 
	 	n^{-3/2}T_n(\gamma, h)\wcn T_{\gamma, h}:=\sup_{0\le s<t\le 1} \frac{|\cU_h(t)-\cU_h(s)|}{[(t-s)(1-(t-s))]^\gamma}.
	 	$$
	 \end{teo}

	 \begin{proof} Set for $f\in \cH^o_\gamma[0, 1]$, and $0\le s<t\le 1$, 
	 	$$
	 	I(f; s, t) := \frac{|f(t) - f(s) - (t - s)(f(1)-f(0))|}{\rho_\gamma(t - s)}.
	 	$$
	 Consider the functions
	 $$
	 F_n(f):=\max_{0\le k<m\le n}I(f; k/n, m/n),\ \ \textrm{and}\ \ F(f)=\sup_{0\le s<t\le 1}I(f; s, t),\ \ f\in \cH^o_\gamma[0, 1].
	 $$
	 Since $\cU_h(0)=\cU_h(1)$ we see that $F(\cU_h)=T_\gamma$. 
	 We have due to anti-symmetry of $h$, for any $0\le k<m\le n$, 
	 \begin{align*}
	 \cU_{h, n}(m/n)-\cU_{h, n}(k/n)&=\sum_{i=1}^m\sum_{j=m+1}^n h(X_i, X_j)-\sum_{i=1}^k\sum_{j=k+1}^n h(X_i, X_j)\\
	 &=\sum_{i=k+1}^m\sum_{j=m+1}^n h(X_i, X_j)+\sum_{i=1}^k\Big[\sum_{j=m+1}^n-\sum_{j=k+1}^n\Big]h(X_i, X_j)\\
	 &=
	 \sum_{i=k+1}^m\sum_{j=m+1}^nh(X_i, X_j)-\sum_{i=1}^k\sum_{j=k+1}^m h(X_i, X_j)\\
	 &=\Delta_{h,n}(k, m).
	 \end{align*}
	 Hence, $F_n(n^{-3/2}\cU_{h,n})=n^{-3/2}T_n(\gamma, h)$. We prove next that 
	 \begin{equation}\label{fin-eq}
	 F_n(n^{-3/2}\cU_{h,n}(\cdot))=F(n^{-3/2}\cU_{h,n}(\cdot))+o_P(1).
	 \end{equation}
	 To this aim we apply the following simple lemma (the proof is given in \cite{rackauskas2004}, see. Lemma 13 therein).
	 
	 \begin{lem}\label{discret} Let $(\eta_n)_{n\ge 1}$ be a tight sequence of random elements in the separable Banach space $\eB$ and $g_n$, $g$ be continuous functionals $\eB \to \eR$. Assume that $g_n$ converges 	 	pointwise to $g$ on $\eB$ and that $(g_n)_{n\ge 1}$ is equicontinuous. Then
	 	$$
	 	g_n(\eta_n) = g(\eta_n) + o_P(1).
	 	$$
	 \end{lem} 
	 We check the continuity of the function $F$ first. We have if $t-s\le 1/2$,
	 $\rho_\gamma(t-s)\ge 2^{-\gamma}(t-s)^\gamma$ and this yields
	 \begin{align*}
	 I(f; s,t) &\le 2^\gamma \sup_{0\le s<t\le 1}\frac{|f(t)-f(s)-(t-s)(f(1)-f(0))|}{(t-s)^\gamma} \\
	 &\le 2^{1+\gamma}||f||_\gamma.
	 \end{align*}
	 If $t-s>1/2$ then $1-(t-s)>1-t$ and $1-(t-s)>s$. This yields
	 \begin{align*}
	 I(f; s,t)&\le 
	 2^\gamma\Big\{\frac{|f(t)-f(1)|}{(1-t)^\gamma}+\frac{|f(0)-f(s)|}{s^\gamma}
	 +\frac{(1-(t-s))|f(1)-f(0)|}{(1-(t-s))^\gamma}\Big\}\\
	 &\le 32^\gamma||f||_\gamma.
	 \end{align*}	
	 Hence, $F(f)\le 6||f||_\gamma$ and this yields the continuity since the inequality $|F(f)-F(g)|\le F(f-g)$ can be easily checked. Similarly we have
	 $|F_n(f)-F_n(g)|\le F_n(f-g)\le 32^\gamma ||f-g||_\gamma$, therefore the sequence $(F_n)$ is equicontinuous on $\cH^o_\gamma[0, 1]$. To check the point-wise convergence on $\cH^o_\gamma[0, 1]$ of $F_n$ to $F$, it is enough to show that for each $f\in \cH^o_\gamma[0, 1]$ the function $(s, t) \to I (f; s, t)$ can be extended by continuity to the compact set $T = \{(s, t)\in [0, 1]^2, 0\le s\le t\le 1\}$. As above we get for $t-s<1/2$
	 $I(f; s, t)\le 2^\gamma \omega_\gamma(f; t-s)+ 2^\gamma|f(1)-f(0)|(t-s)^{1-\gamma}$, which allows the
	 continuous extension along the diagonal putting $I(f; s, s):= 0$. If $t-s>1/2$ we get 
	 $I(f; s, t)\le 2^\gamma\omega_\gamma(f, 1-(t-s))+2^\gamma|f(1)-f(0)|(1 + t - s)^{1-\gamma}$ which allows the continuous extension at the point $(0, 1)$ putting $I(f; 0, 1):= 0$.
	 
	 The pointwise convergence of $(F_n)$ being now established, and observing that by the $(\gamma, h)$-FCLT, the sequence $n^{-3/2}U_n$ is tight, Lemma \ref{discret} gives (\ref{fin-eq}). Since $F$ is continuous, continuous mapping theorem together with $(h, \gamma)$-FCLT yield
	 $$
	 F(n^{-3/2}U_{h,n}(\cdot))\wcn F(\cU_h)= T_{\gamma, h}.
	 $$
	 By (\ref{fin-eq}) we get
	 $$
	 n^{-3/2}T_{n}(h, \gamma)=F_n(n^{-3/2}U_{h,n})\wcn  T_{\gamma, h}.
	 $$ 
	 This completes the proof.
	 \end{proof}

Combination of this general result with Theorem \ref{iidsample} and Theorem \ref{teo:mixing} gives the proofs of Theorem \ref{theoremA}  and Theorem \ref{theoremB} respectively.

\section{Behavior under the alternative}\label{sec:alt}
	
To discuss the behaviour of the test statistics $T_{n}(\gamma, h)$ under the alternative we assume that for each $n\ge 1$ we have two probability measures $P_n$ and $Q_n$ on $(\eS, \cS)$ and a random sample $(X_{ni})_{1\le i\le n}$ such that for $k^*_n, m^*_n\in \{1, \dots, n\}$, 
$$
P_{X_{ni}}=\begin{cases} Q_n, \ \ &\textrm{for}\ \ i\in I^*:=\{k^*_n+1, \dots, m^*_n\}\\
P_n, \ \ &\textrm{for}\ \ i\in I_n\setminus I^*.
\end{cases}
$$
We will write $k^\star=k_n^\star$, $m^\star=m_n^\star$ and $\ell^\star=m^\star-k^\star$ for short. Set
$$
\delta_n=\delta(P_n, Q_n)=\int_{\eS}\int_{\eS} h(x, y)Q_n(dx)P_n(dy). 
$$  
Note that $\delta_n$ measures in a sense the difference between the probability distributions $P_n$ and $Q_n$. If $P_n=Q_n$, then $\delta_n=0$. If $h(x, y)=h_c(x, y)$ then $\delta_n=\int sP_n(dx)-\int xQ_n(dx)$. If $h=h_W$ then $\delta_n=\int P_n(x)Q_n(dx)-\int Q_n(x)P_n(dx)$. The general consistency result is in the following elementary lemma.  	 
\begin{lem}\label{lem-consistency}
If 
\begin{equation}\label{cons:eq:1}
\frac{1}{\rho_\gamma(\ell^*/n)}n^{-3/2}\sum_{i\in I^*}\sum_{j\in I_n\setminus I^*}\big[h(X_{ni}, X_{nj})-\delta_n\big]=O_P(1) 
\end{equation}
and 
\begin{equation}\label{cons:eq:2}
\sqrt{n}\big|\delta_n\big|\Big[\frac{\ell^*}{n}\Big(1-\frac{\ell^*}{n}\Big)\Big]^{1-\gamma}\to \infty,
\end{equation}
then
$$
n^{-3/2}T_n(\gamma, h)\pc \infty.
$$
\end{lem}

\begin{proof}[Proof of Theorem \ref{theoremC}] Set for $i\in I^*, j\in I_n\setminus I^*$, 
	$$
	Z_{ij}=h(X_{ni}, X_{nj})-\delta_n.
	$$
Noting that $EZ_{ij}=0$ and $EZ^2_{ij}= \nu_n$ for any $i\in I^*, j\in I_n\setminus I^*$, we obtain
$$
E\Big(\sum_{i\in I^*}\sum_{j\in I_n\setminus I^*}Z_{ij}\Big)^2=\sum_{i,i'\in I^*}\sum_{j,j'\in I_n\setminus I^*}E(Z_{ij}Z_{i'j'})\le n\ell^*(n-\ell^*)\nu_n.
$$
This yields (\ref{cons:eq:1}) by (\ref{cond:consis:1}) and completes the proof. 
\end{proof}

\begin{proof}[Proof of Theorem \ref{theoremD}]
%
We will use a Hoeffding decomposition adjusted to the changing distribution. 
To this aim we define
\begin{align*}
h_{1,n}(x)&:=\int_{\eS}h(x, y)Q_n(dy)-\delta_n,\\ 
h_{2,n}(y)&:=\int_{\eS}h(x, y)P_n(dx)-\delta_n,\\ 
g_n(x, y) &:= h(x,y)-h_{1,n}(x) - h_{2,n}(y) - \delta_n. 
\end{align*}

Next we show that the following estimates hold with an absolute constant $C>0$:
	\begin{align}
	E\left[\bigg(\sum_{i=k^*+1}^{k^*+\ell^*}h_{2,n}(X_{i,n})\bigg)^2\right]&\leq C\ell^*,\label{ConAlt1}\\
	E\left[\bigg(\sum_{i=1 }^{k^*} h_{1,n}(X_{i,n})+\sum_{i=k^*+\ell^*+1 }^n h_{1,n}(X_{i,n})\bigg)^2\right]&\leq C(n-\ell^*),\label{ConAlt2}\\
	E\left[\bigg(\sum_{i=1 }^{k^*}\sum_{j=k^*+1}^{k^*+\ell^*}g_n(X_{i,n},X_{j,n})+\sum_{i=k^*+\ell^*+1 }^n \sum_{j=k^*+1}^{k^*+\ell^*}g_n(X_{i,n},X_{j,n})\bigg)^2\right]&\leq C\ell^*(n-\ell^*)\label{ConAlt3}.
	\end{align}
These estimates yield 
$$
E\Big(\sum_{i\in I^*}\sum_{j\in I_n\setminus I^*}[h(X_{ni}, X_{nj})-\delta_n] \Big)^2\le Cn\ell^*(n-\ell^*)
$$	
with an absolute constant $C>0$ and (\ref{cons:eq:1}) follows by (\ref{cond:consis:2}). Hence, it remains to prove (\ref{ConAlt1})--(\ref{ConAlt3}). 

Conditions (\ref{ConAlt1}) and (\ref{ConAlt2}) follow from Lemma \ref{beta:Ros}, (\ref{ConAlt3}) follows from Lemma \ref{g2b}.
\end{proof}

\section{Critical Values}\label{sec:crit}

Below in Table \ref{tab1}, we give the upper quantiles of limit distribution of the one-sided and two-sided test statistics, that is
\begin{align*}
T_1&:=\sup_{s,t\in[0,1],s<t}\frac{B(t)-B(s)}{(t-s)^{\gamma}(1-(t-s))^\gamma}\\
T_2&:=\sup_{s,t\in[0,1],s<t}\frac{\big|B(t)-B(s)\big|}{(t-s)^{\gamma}(1-(t-s))^\gamma},
\end{align*}
where $B$ is a standard Brownian bridge. The distribution was evaluated on a grid of size 10,000 and we run a Monte-Carlo-simulation with 30,000 runs.

\begin{table}[ht]
\caption{Upper quantiles of $T_1$ (upper half) and $T_2$ (lower half).}\label{tab1}
\resizebox{\columnwidth}{!}{
\begin{tabular}{|l||l|l|l|l|l|l|l|l|l|}
\hline 
 & \quad\quad\quad\quad & \quad\quad\quad\quad & \quad\quad\quad\quad & \quad\quad\quad\quad & \quad\quad\quad\quad & \quad\quad\quad\quad & \quad\quad\quad\quad & \quad\quad\quad\quad & \quad\quad\quad\quad \\  
 & 50\% & 20\% & 10\% & 5\% & 2.5\% & 1\% & 0.5\% & 0.25\% & 0.1\% \\ 
  &  &  &  &  &  &  &  &  &  \\  
\hline
\hline
one-sided &  &  &  &  &  &  &  &  &  \\ 
$\gamma=0$ & 1.101 & 1.360 & 1.515 & 1.647 & 1.769 & 1.922 & 2.029 &  2.121 & 2.244  \\
 &  &  &  &  &  &  &  &  &  \\  
$\gamma=0.05$ & 1.199 & 1.466 & 1.631 & 1.770 & 1.899 & 2.041 & 2.161 & 2.251 & 2.370  \\ 
 &  &  &  &  &  &  &  &  &  \\ 
$\gamma=0.1$ & 1.230 & 1.591 & 1.764 & 1.914 & 2.045 & 2.211 & 2.324 &  2.437 & 2.567  \\
 &  &  &  &  &  &  &  &  &  \\  
$\gamma=0.15$ & 1.416 & 1.712 & 1.897 & 2.057 & 2.213 & 2.379 & 2.505 & 2.611 & 2.760  \\ 
 &  &  &  &  &  &  &  &  &  \\ 
$\gamma=0.2$ & 1.551 & 1.871 & 2.061 & 2.231 & 2.387 & 2.571 & 2.695 & 2.817 & 2.999 \\ 
 &  &  &  &  &  &  &  &  &  \\ 
$\gamma=0.25$ &  1.705 & 2.033 & 2.232 & 2.411 & 2.569 & 2.757 & 2.906 & 3.031 & 3.169  \\ 
 &  &  &  &  &  &  &  &  &  \\ 
$\gamma=0.3$ &  1.903 & 2.238 & 2.445 & 2.623 & 2.784 & 3.00 & 3.167 & 3.309 & 3.444  \\ 
 &  &  &  &  &  &  &  &  &  \\ 
$\gamma=0.35$ &  2.148 & 2.475 & 2.687 & 2.880 & 3.069 & 3.271 & 3.419 & 3.578 & 3.753  \\ 
 &  &  &  &  &  &  &  &  &  \\ 
$\gamma=0.4$ &  2.508 & 2.814 & 3.015 & 3.192 & 3.367 & 3.581 & 3.717 & 3.850 & 4.023 \\ 
 &  &  &  &  &  &  &  &  &  \\ 
$\gamma=0.45$ &  3.121 & 3.387 & 3.560 & 3.723 & 3.877 & 4.079 & 4.223 & 4.388 & 4.585  \\ 
 &  &  &  &  &  &  &  &  &  \\ 
\hline 
two-sided &  &  &  &  &  &  &  &  &  \\ 
$\gamma=0$ & 1.213 & 1.460 & 1.612 & 1.741 & 1.857 & 2.012 & 2.104 & 2.195 & 2.311   \\
 &  &  &  &  &  &  &  &  &  \\  
$\gamma=0.05$ & 1.314 & 1.573 & 1.732 & 1.862 & 1.987 & 2.143 & 2.242 & 2.334 & 2.476 \\ 
 &  &  &  &  &  &  &  &  &  \\ 
$\gamma=0.1$ &  1.423 & 1.708 & 1.876 & 2.016 & 2.148 & 2.306 & 2.417 & 2.501 & 2.638 \\
 &  &  &  &  &  &  &  &  &  \\  
$\gamma=0.15$ &  1.544 & 1.839 & 2.017 & 2.172 & 2.309 & 2.485 & 2.596 & 2.720 & 2.859  \\ 
 &  &  &  &  &  &  &  &  &  \\ 
$\gamma=0.2$ &  1.684 & 1.995 & 2.175 & 2.344 & 2.498 & 2.677 & 2.812 & 2.947 &  3.085  \\ 
 &  &  &  &  &  &  &  &  &  \\ 
$\gamma=0.25$ &  1.846 & 2.172 & 2.357 & 2.527 & 2.684 & 2.862 & 2.998 & 3.104 & 3.219 \\ 
 &  &  &  &  &  &  &  &  &  \\ 
$\gamma=0.3$ &  2.042 & 2.372 & 2.572 & 2.748 & 2.906 & 3.122 & 3.262 & 3.372 & 3.562  \\ 
 &  &  &  &  &  &  &  &  &  \\ 
$\gamma=0.35$ &  2.294 & 2.621 & 2.825 & 3.017 & 3.196 & 3.387 & 3.541 & 3.695 & 3.860 \\ 
 &  &  &  &  &  &  &  &  &  \\ 
$\gamma=0.4$ &  2.654 & 2.961 & 3.150 & 3.330 & 3.499 & 3.695 & 3.852 & 4.002 & 4.208 \\ 
 &  &  &  &  &  &  &  &  &  \\ 
$\gamma=0.45$ &  3.268 & 3.529 & 3.697 & 3.852 & 4.011 & 4.216 & 4.362 & 4.486 & 4.627  \\ 
 &  &  &  &  &  &  &  &  &  \\ 
\hline 
\end{tabular}
}
\end{table}

\small

\section*{Acknowlegdements} We thank the editor and two anonymous referees for their careful reading of the article and their thoughtful comments, which have led to an improvement of the article. We also would like to thank Lea Wegner for her help concerning language corrections and additional R programming.

\end{document}